\documentclass{amsart}

\usepackage{url}

\newtheorem{theorem}{Theorem}[section]
\newtheorem{proposition}[theorem]{Proposition}
\newtheorem{lemma}[theorem]{Lemma}
\newtheorem{corollary}[theorem]{Corollary}
\newtheorem{fact}[theorem]{Fact}

\theoremstyle{definition}
\newtheorem{definition}[theorem]{Definition}

\theoremstyle{remark}
\newtheorem{remark}[theorem]{Remark}
\newtheorem{example}[theorem]{Example}
\newtheorem{question}[theorem]{Question}

\numberwithin{equation}{section}

\def\Ind{\setbox0=\hbox{$x$}\kern\wd0\hbox to 0pt{\hss$\mid$\hss} \lower.9\ht0\hbox to 0pt{\hss$\smile$\hss}\kern\wd0} 

\def\Notind{\setbox0=\hbox{$x$}\kern\wd0\hbox to 0pt{\mathchardef \nn=12854\hss$\nn$\kern1.4\wd0\hss}\hbox to 0pt{\hss$\mid$\hss}\lower.9\ht0 \hbox to 0pt{\hss$\smile$\hss}\kern\wd0}

\def \d {\delta}
\def \dd {\partial}

\def \D {\Delta}

\def \la {\langle}
\def \ra {\rangle}

\def \N {\mathbb N}

\def \P {\mathcal P}
\def \ord {\operatorname{ord}}
\def \n {\mathfrak n}
\def \Nn {\N^m\times \n}
\def \lUB {\operatorname{LUB}}
\def \l {\mathfrak l}
\def \Vol {\operatorname{Vol}}
\def \L {\Lambda}
\def \M {\mathfrak M}

\title[Estimates for differential dimension polynomials]{Estimates for the coefficients of differential \\ dimension polynomials}
\author{Omar Le\'on S\'anchez}
\address{Omar Le\'on S\'anchez\\
University of Manchester\\
School of Mathematics\\
Oxford Road \\
Manchester, M13 9PL.}
\email{omar.sanchez@manchester.ac.uk}
\date{\today}
\subjclass[2010]{12H05, 14Q20}
\keywords{Kolchin polynomial, typical differential dimension, Hilbert-Samuel functions}

\begin{document}

\begin{abstract}
We answer the following long-standing question of Kolchin: given a system of algebraic-differential equations $\Sigma(x_1,\dots,x_n)=0$ in $m$ derivatives over a differential field of characteristic zero, is there a computable bound, that \emph{only depends on the order of the system} (and on the fixed data $m$ and $n$), for the typical differential dimension of any prime component of $\Sigma$? We give a positive answer in a strong form; that is, we compute a (lower and upper) bound for \emph{all} the coefficients of the Kolchin polynomial of every such prime component. We then show that, if we look at those components of a specified differential type, we can compute a significantly better bound for the typical differential dimension. This latter improvement comes from new combinatorial results on characteristic sets, in combination with the classical theorems of Macaulay and Gotzmann on the growth of Hilbert-Samuel functions.
\end{abstract}

\maketitle

\tableofcontents

\section{Introduction}
The role of numerical polynomials in commutative algebra was initiated by Hilbert's theorem on the growth of the dimension of the homogeneous components of a graded module over a polynomial ring. More precisely, Hilbert showed that given a graded module $M=\bigoplus_s M_s$ over $K[x_1,\dots,x_n]$, there is a numerical polynomial $p(t)$ such that, for $s$ sufficiently large, $p(s)=\dim_KM_s$. This is a generalization of the well known fact that the number of monomials of degree $d$ in $m$ variables equals $\binom{m-1+d}{d}$; note that the latter is a numerical polynomial in $d$.

The corresponding role of numerical polynomials in differential algebra was initiated by Kolchin (and further developed by Johnson and Sit, among several others). Let us briefly recall the relevant objects. Let $(K,\D=\{\d_1,\dots,\d_m\})$ be a differential field of characteristic zero with $m$ commuting derivations, and denote by $K\{x_1,\dots,x_n\}$ the (differential) ring of differential polynomials in $n$ variables over $K$. Given a prime differential ideal $\P$ of $K\{x_1,\dots,x_n\}$, if we set $a=(a_1,\dots,a_n)$ to be a generic point of $\P$ (for instance, choose $a_i=x_i+\P\in K\{x_1,\dots,x_n\}/\P$), then Kolchin showed \cite[Chapter II]{Kolchin} that there is a numerical polynomial $\omega_\P(t)$ such that for $s$ sufficiently large
$$\omega_\P(s)=\operatorname{trdeg}_K K(\d_1^{u_1}\cdots\d_m^{u_m}a_i:\; 1\leq i\leq n,\, u_1+\cdots+u_m\leq s).$$
This numerical polynomial is called the Kolchin polynomial (or differential dimension polynomial) of $\P$. If $\P$ is generated (as a radical differential ideal) by $f_1,\dots,f_\ell$, then, as explained by Mikhalev and Pankratev in \cite{MP}, the Kolchin polynomial $\omega_\P$ is an algebraic form of Einstein's notion of \emph{strength} of the system $f_1=\cdots=f_\ell=0$. Roughly speaking, using Einstein's terminology \cite{Ein}, one can say that, for $s$ sufficiently large, $\omega_\P(s)$ equals the number of free Taylor coefficients of order at most $s$ of a generic solution of the system.

The polynomial $\omega_\P$ carries two important differential birational invariants: its degree, called the differential type of $\P$ and denoted by $\tau$, and its leading coefficient (as a numerical polynomial, see \eqref{numform} below), called the typical differential dimension of $\P$ and denoted by $a_\tau$. Let us consider a simple, but illustrative, example:

\begin{example}\label{introex}
Let $n=1$ and $m=2$; that is, we work in the context of systems in one variable and derivations $\D=\{\d_1,\d_2\}$. Given a nonnegative integer $r$, consider the system $\d_1^rx=\d_2^rx=0$. Note that the differential ideal $\P$ generated by $\d_1^r x$ and $\d_2^rx$ is prime. Also, given a generic solution $a$, for $s\geq 2r$ we have
$$\operatorname{trdeg}_K K(\d_1^{u_1}\d_2^{u_2}a:\, u_1+u_2\leq s)=r^2.$$
Thus, in this case, $\omega_\P(t)$ is constant equal to $r^2$. In particular, the differential type of $\P$ is zero and its typical differential dimension is $r^2$.
\end{example}

The main question we consider in this paper (and originally studied by Kolchin) is the following.

\begin{question}\label{introquestion}
Given nonnegative integers $r,m,n, \tau$, is there a computable function $B(r,m,n,\tau)$ such that for any differential field $(K,\D=\{\d_1,\dots,\d_m\})$ of characteristic zero and any differential system $\Sigma\subset K\{x_1,\dots,x_n\}$ of order $r$, if $\P$ is a prime component of $\Sigma$ of differential type $\tau$, then the typical differential dimension of $\P$ satisfies
$$a_\tau\leq B(r,m,n,\tau).$$
\end{question}

Kolchin devoted a section of his book \cite[Chapter IV, \S17]{Kolchin} on this question. While he does not give a definitive answer, he is able to show the following (partial) result.

\begin{fact}\label{introfact}
Using the notation of Question \ref{introquestion}, if $\tau=m$ then $a_\tau\leq n$, and if $\tau=m-1$ then $a_\tau\leq nr$. These bounds are clearly optimal (in the sense that one can find a system $\Sigma$ where equality holds).
\end{fact}

Moreover, in that same section of his book, he conjectures that (with no assumption on $\tau$)
$$a_\tau\leq n\cdot\binom{r-1+m-\tau}{m-\tau}.$$
However, this bound does not generally hold. Indeed, in Example \ref{introex} we saw that there are systems, in two derivations and one variable, of differential type zero with typical differential dimension $r^2$. Kolchin's conjectural bound yields $\frac{r(r+1)}{2}$, which is less than $r^2$ when $r$ is larger than one\footnote{It was pointed out to me, by Alexander Levin, that Kolchin was aware that his conjectural bound is incorrect.}.

There has been some attempts to answer Kolchin's question (Question \ref{introquestion}). For example, in \cite{Kon}, Kondratieva makes the additional assumption that the degree of the system $\Sigma$ is also bounded by $r$ and then shows that if $\tau\leq m-2$ we have 
$$a_\tau\leq 2^{2^{2^{m-\tau-2}}}(2A(m+7,2^{9r}))^{(m+2)2^{m-\tau-2}}$$
where $A$ denotes the Ackermann function. In the case when the system is linear, in \cite{Gri} Grigoriev proved
$$a_\tau\leq n(4m^2nr)^{4^{m-\tau-1}(2(m-\tau))}$$
There are two disadvantages of these bounds. First, they depend on the degree of the system, whereas Kolchin's question asks for a bound that only depends on the order. Second, even for small values of $m$ they are highly not optimal. For instance, consider the case $m=2$, $n=1$, and $\tau=0$; Kondratieva's bound involves a term of the form $A(9,2^{9r})$, which yields extremely large values, while Grigoriev's bound yields $(16\,r)^{16}$. These are both far from the optimal value since, as we will see in Corollary \ref{final2}, in this case $a_\tau\leq r^2$ (and this is optimal by Example \ref{introex}).

The goal of this paper is to show that Question \ref{introquestion} has a positive answer, and we obtain two different ways of constructing appropriate bounds. The first construction appears in Section \ref{rough} where we answer the question in a stronger form; that is, we prove that there is a computable (recursive) bound, that only depends on the order of the differential system $\Sigma$, for the sum of the absolute values of the (standard) coefficients of the Kolchin polynomial of any prime component $\P$ of $\Sigma$. In particular, this yields a bound for the typical differential dimension, and so this already answers Kolchin's question positively. We note that this bound does not depend on the differential type of $\P$. In the case when $m=2$, $n=1$ and $\tau=0$, this bound yields $a_0\leq 4r^2(2r+1)^2$, but as we noted at the end of the above paragraph one can in fact show that in this case $a_0\leq r^2$. To prove the latter, we perform a second construction of a bound, that now depends on the differential type, for the typical differential dimension. In Section \ref{several} we build (recursively) a sequence $\bar\mu$, depending on the data $(r,m,n,\tau)$, inside of $\N^m\times\{1,\dots,n\}$ with the property that its Hilbert-Samuel function bounds from above the Hilbert-Samuel function associated to any prime component with differential type $\tau$ of a differential system of order $r$. The proof of this fact uses a new result on the combinatorial structure of characteristic sets of prime differential ideals, established in Section \ref{combi}, together with the well known theorems of Macaulay and Gotzmann on the growth of the Hilbert-Samuel function. Finally, in Section \ref{compu}, we use the algorithmic construction of the sequence $\bar\mu$ to give explicit formulas to compute bounds for the differential type, and we perform these computations for the cases $m=2$ and $3$.

It is worth mentioning that J. Freitag has recently (re-)established some bounds and finiteness principles in differential algebra \cite{James}, using results from the theory of well quasi-orderings. One of the results there can be considered as a partial (noneffective) converse to our main result here. More precisely, he shows that given a numerical polynomial $p(t)$ there is a bound, that depends only on the coefficients of $p(t)$ (and the number of derivations and variables), for the order of a characteristic set of a prime differential ideal with Kolchin polynomial $p(t)$.
 
\bigskip

\noindent { \it Acknowledgements.} I am grateful to the organizers of the Differential Algebra Workshop April 2016 (supported by CUNY CIRG \#2248) for providing a very stimulating and friendly environment where part of this work was initiated.

\section{Notation and preliminaries}\label{pre}

In this section we fix the notation that will be used throughout the paper. For the convenience of the reader, we give an account of the classical and recent results that will be used in subsequent sections. Fix a differential field $(K,\D=\{\d_1,\dots,\d_m\})$ of characteristic zero with $m$ commuting derivations. Henceforth, the differential terminology will be with respect to our distinguished set $\D$ of commuting derivations. 

For an $m$-tuple $\xi=(u_1,\dots,u_m)\in \N^m$, we define the order of $\xi$ by 
$$\ord\xi=u_1+\cdots+u_m$$
and use multi-index notation to denote derivative operators $\d^\xi:=\d_1^{u_1}\cdots\d_{m}^{u_m}$. Fix an $n$-tuple $x=(x_1,\dots,x_n)$ of differential indeterminates. We denote the (differential) ring of differential polynomials in variables $x$ over $K$ by
$$K\{x\}:=K[\d^\xi x_i: 1\leq i\leq n, \; \xi\in \N^m].$$

Recall that a polynomial $p\in \mathbb Q[t]$ is \emph{numerical} if $p(s)\in \mathbb Z$ for all integers $s$. Such polynomials always have the form
\begin{equation}\label{numform}
p(t)=\sum_{i=0}^da_i\binom{t+i}{i}, \quad a_i\in \mathbb Z.
\end{equation}
We will sometimes refer to the $a_i$'s as the \emph{standard coefficientes} of $p$.

Given a prime differential ideal $\P$ of $K\{x\}$ and a generic point $a=(a_1,\dots,a_n)$ of $\P$ (for instance, one can take $a$ to be the image of $x$ under the canonical map $K\{x\}\to K\{x\}/\P$), the \emph{Kolchin polynomial of $\P$} (or differential dimension polynomial) is the unique numerical polynomial $\omega_\P(t)$ such that for sufficiently large $s\in \N$ 
$$\omega_{\P}(s)=\operatorname{trdeg}_KK(\d^\xi a_i:1\leq i\leq n, \; \ord\xi\leq s)$$
See \cite[Chapter II]{Kolchin} for details on this. For our purposes, it suffices to note that $\deg \omega_\P\leq m$ and that the standard coefficient $a_m$ of $\omega_\P$ equals the differential transcendence degree, over $K$, of $K\la a\ra$ (the differential field generated by $a$ over $K$). Additionally, the Kolchin polynomial has two important differential birational invariants: the \emph{differential type} of $\P$ is defined as the degree of $\omega_\P$ and the \emph{typical differential dimension of $\P$} is the standard leading coefficient of $\omega_\P$.

We now wish to state the (well known) fact that the Kolchin polynomial of $\P$ can be computed from the leaders of a characteristic set of $\P$. To do this, let us recall that the canonical orderly ranking $\unlhd$ on the set of (algebraic) indeterminates $\{\d^\xi x_i:1\leq i\leq n,\; \xi\in \N^m\}$ is defined as follows: $\d_1^{e_1}\cdots\d_m^{e_m}x_i\unlhd\d_1^{e_1'}\cdots\d_m^{e_m'}x_j$ if and only if
$$(\sum_k e_k,i,e_1,\dots,e_m)\leq_{\text{lex}} (\sum_k e_k',j,e_1',\dots,e_m')$$
where $\leq_{\text{lex}}$ denotes the (left) lexicographic order. The \emph{leader} of a differential polynomial $f\in K\{x\}\setminus K$ is the highest $\d^\xi x_i$ that appears in $f$ with respect to $\unlhd$, and the order of $f$ is simply the order of its leader. Given any finite collection of differential polynomials $\Sigma \subset K\{x\}\setminus K$, by a leader of $\Sigma$ we mean a leader of one of its elements and by the order of $\Sigma$ we mean the maximum order among its elements. As we would like to avoid certain technicalities that are unnecessary for our purposes, we will not give the proper definition of a characteristic set. Let us just say that, for us, a \emph{characteristic set} of the prime differential ideal $\P\subset K\{x\}$ is a finite subset of $\P$ which is ``reduced'' and ``minimal'' with respect to the canonical orderly ranking $\unlhd$. We refer the reader to \cite[Chapter I]{Kolchin} for further details. 

We now recall the notion of volume of lattices of $\N^m$. We let $\leq$ denote the product order on $\N^m$; that is, $(u_1,\dots,u_m)\leq (v_1,\dots,v_m)$ means that $u_i\leq v_i$ for $i=1,\dots,m$. Given any $E\subseteq \N^m$ and a nonnegative integer $s$, the \emph{volume of $E$} at level $s$ is
$$V_E(s)=\{\xi\in \N^m: \ord\xi\leq s \text{ and } \xi\not\geq \eta \text{ for all }\eta\in E\}.$$
In \cite[Chapter 0, \S17]{Kolchin}, Kolchin shows that for any $E\subseteq \N^m$ there is a numerical polynomials $\omega_E(t)$ such that for sufficiently large $s\in \N$
$$\omega_E(s)=|V_E(s)|.$$
Furthermore, $\deg \omega_E\leq m$; equality occurs if and only if $E$ is empty, in which case 
\begin{equation}\label{maxi}
\omega_E(t)=\binom{t+m}{m}.
\end{equation}

We can now state the following result of Kolchin's \cite[Chapter II, \S12]{Kolchin}. 
\begin{fact}\label{factKol}
Let $\P$ be a prime differential ideal of $K\{x\}$ and $\L$ a characteristic set of $\P$. If for each $1\leq i\leq n$ we denote by $E_i$ the set of all $\xi\in\N^m$ such that $\d^\xi x_i$ is a leader of $\L$, then
$$\omega_\P(t)=\sum_{i=1}^n \omega_{E_i}(t).$$
\end{fact}

We now wish to recall Macaulay's theorem and Gotzmann's persistence theorem on the growth of Hilbert-Samuel functions (for details we refer the reader to \cite[Chapter IV]{BH}). Given positive integers $a$ and $d$, one can uniquely write $a$ in the form
\begin{equation}\label{rep22}
a=\binom{k_d}{d}+\binom{k_{d-1}}{d-1}+\cdots\binom{k_j}{j}
\end{equation}
where $k_d>k_{d-1}>\cdots>k_j\geq j\geq 1$. This is called the \emph{$d$-binomial representation of $a$}. Using this representation, one defines the Macaulay's function $*^{\la d\ra}:\N\to \N$ by setting
$$a^{\la d\ra}:=\binom{k_d+1}{d+1} +\binom{k_{d-1}+1}{d}+\cdots + \binom{k_j+1}{j+1}$$
and $0^{\la d\ra}=0$. It follows from the definition that 
\begin{equation}\label{ineq11}
a< b \implies  a^{\la d\ra}<b^{\la d\ra} 
\end{equation}

We recall that a subset $M\subseteq \N^m$ is said to be \emph{compressed} if for all $\xi,\eta\in \N^m$ with $\ord\xi=\ord\eta$ we have
$$\left(\eta\unrhd \xi \text{ and } \xi\geq \tau \text{ for some }\tau\in M\right)\implies (\eta\geq \zeta \text{ for some } \zeta\in M).$$
The \emph{Hilbert-Samuel function} $H_M:\N\to\N$ of $M$ is defined as
$$H_M(d)=|\{\xi\in \N^m: \ord \xi=d \text{ and } \xi\not\geq \eta \text{ for all }\eta\in M\}|.$$

\begin{theorem}[Macaulay's theorem]\label{Macaulay}
For any $M\subseteq \N^m$, and positive integer $d$, we have
$$H_M(d+1)\leq H_M(d)^{\la d\ra}$$
Moreover, if $M$ is compressed and $\ord\xi\leq d$ for all $\xi\in M$, then
$$H_M(d+1)=H_M(d)^{\la d\ra}.$$
\end{theorem}

\begin{theorem}[Gotzmann's persistence theorem]\label{Gotzmann}
Let $M\subset \N^m$. If there is a positive integer $d$ such that $\ord\xi\leq d$ for all $\xi\in M$ and $H_M(d+1)=H_M(d)^{\la d\ra}$, then $H_M(s+1)=H_M(s)^{\la s\ra}$ for all $s\geq d$.
\end{theorem}
 
Macaulay's theorem gives a \emph{sufficient} condition for the Hilbert-Samuel function to have maximal growth, we will also make use of the following theorem that gives a \emph{necessary} condition to have maximal growth (which appears as Corollary 4.9 in \cite{GO}). For $\xi=(u_1,\dots,u_m)$ and $\eta=(v_1,\dots,v_m)$ in $\N^m$, we let 
$$\lUB(\xi,\eta)=(\max\{u_1,v_1\},\dots,\max\{u_m,v_m\});$$
in other words, $\lUB(\xi,\eta)$ is the least upper bound of $\xi$ and $\eta$ with respect to the product order.

\begin{theorem}\label{HS}
Let $M\subseteq \N^m$ and $d>1$ an integer. If $H_M(d)=H_M(d-1)^{\la d-1\ra}$, then for any distinct $\xi,\xi'\in M$ of order at most $d-1$ there exists a sequence $\xi=\eta_1,\eta_2,\dots,\eta_s=\xi'$ of $M$ with each $\eta_i<\lUB(\xi,\xi')$ and of order at most $d-1$ such that
$$\ord\lUB(\eta_i,\eta_{i+1})\leq d, \quad \text{ for all }i=1,\dots,s-1.$$
\end{theorem}

In Section \ref{several} we will make use of the following series of technical (combinatorial) results on Macaulay's function. They are useful to deal with Hilbert-Samuel functions in several copies of $\N^m$; which will appear in Theorem \ref{main}. 
It is worth recalling at this point that for $c$ and $d$ positive integers we have the classical binomial identity
\begin{equation}\label{bin11}
\binom{c+d}{d}=\binom{c-1+d}{d}+\binom{c-1+d}{d-1}.
\end{equation}


In what follows $d$ denotes a positive integer. Note that, by  \eqref{bin11} and \eqref{rep22}, we can always write $a>0$ uniquely in the form 
\begin{equation}\label{form11}
a=\binom{c-1+d}{d}+A \quad \text{ with } 0\leq A < \binom{c-1+d}{d-1}
\end{equation}
for some $c>0$. It follows from the definition that
$$a^{\la d\ra}=\binom{c+d}{d+1}+A^{\la d-1\ra}$$

Moreover, we can also write $a$ uniquely in the form
\begin{equation}\label{form22}
a=\binom{c-1+d}{d}+A \quad \text{ with } 0< A \leq \binom{c-1+d}{d-1}
\end{equation}
for some $c\geq 0$. Indeed, by \eqref{form11}, we only need to consider the case when $a=\binom{b-1+d}{d}$ for some $b> 0$. But now, by \eqref{bin11}, we can write
$$a=\binom{b-2+d}{d}+\binom{b-2+d}{d-1}$$
obtaining the desired shape (with $c=b-1$). (Note that in may happen that $b=1$ in which case the first term on the left-hand-side becomes $\binom{d-1}{d}:=0$.) Also, by \eqref{bin11}, in this case we get the identity
$$a^{\la d\ra}=\binom{b-1+d}{d+1}+\binom{b-1+d}{d}.$$
Therefore, regardless of the form we write $a> 0$, either as in \eqref{form11} or \eqref{form22}, we get
\begin{equation}\label{use00}
a^{\la d\ra}=\binom{c+d}{d+1}+A^{\la d-1\ra}.
\end{equation}

\begin{lemma}\label{rep}
Let $m$ and $d$ be a positive integers. If $a$ and $b$ are nonnegative integers, we have
\begin{enumerate}
\item $a^{\la d\ra}+b^{\la d\ra}\leq (a+b)^{\la d\ra} $
\item If, additionally, $a$ and $b$ are at most $\binom{m-1+d}{d}$ and $a+b\leq \binom{m-1+d}{d}+c$ for some integer $c\geq 0$, then
$$a^{\la d\ra}+b^{\la d\ra}\leq \binom{m-1+d}{d}^{\la d\ra}+c^{\la d\ra}$$
Furthermore, if $a,b>0$ and $c=0$, then 
$$a^{\la d\ra}+b^{\la d\ra}< \binom{m-1+d}{d}^{\la d\ra}$$
\end{enumerate}
\end{lemma}
\begin{proof}
We assume $a\geq b$. By \eqref{ineq11}, we may also assume that $b>0$ and, in (2), that 
$$a<\binom{m-1+d}{d} \quad \text{ and }\quad  a+b= \binom{m-1+d}{d}+c.$$

We proceed by induction on $d$. For the base case, $d=1$, note that 
$$a^{\la 1\ra}=\frac{a(a+1)}{2}.$$ 
Thus,
$$a^{\la 1\ra}+b^{\la 1\ra}\leq \frac{a(a+1)+b(b+1)+2ab}{2}=\frac{(a+b)(a+b+1)}{2}=(a+b)^{\la 1\ra}$$
To prove (2) in this case write $m=a+i=b+j$ with $i,j\geq 0$, then
$$
m c=(a+i)(a-j)=a(a+i-j)-ij=ab-ij\leq ab,
$$
and so
$$a^{\la 1\ra}+b^{\la 1\ra}=\frac{(a+b)^2-2ab+a+b}{2}\leq \frac{(m+c)^2-2mc+m+c}{2}=m^{\la 1\ra}+c^{\la 1\ra}$$
For the 'furthermore' clause, note that if $c=0$ then $-2ab<-2mc$, and hence the above inequality becomes strict in this case.

We now assume $d>1$. By \eqref{form11} and $\eqref{form22}$ we can write
$$a=\binom{s-1+d}{d}+A \quad \text{ with } 0\leq A < \binom{s-1+d}{d-1}$$
where $s>0$ and 
$$b=\binom{t-1+d}{d}+B \quad \text{ with } 0< B \leq \binom{t-1+d}{d-1}$$
where $t\geq 0$. We now construct new integers $a_1$ and $b_1$ such that
\begin{equation}\label{one}
0\leq b_1< b\leq a<a_1\leq \binom{s+d}{d}
\end{equation}
\begin{equation}\label{two}
a+b=a_1+b_1
\end{equation}
\begin{equation}\label{three}
a^{\la d\ra}+b^{\la d\ra}\leq a_1^{\la d\ra}+b_1^{\la d\ra}
\end{equation}
We consider two cases:

\smallskip
\noindent  {\bf Case 1:} $A+B<\binom{s-1+d}{d-1}$. In this case we set 
$$a_1=\binom{s-1+d}{d}+A+B\quad \text{ and }\quad b_1=\binom{t-1+d}{d}$$
Clearly \eqref{one} and \eqref{two} are satisfied. By \eqref{use00}, we have
$$a^{\la d\ra}=\binom{s+d}{d+1}+A^{\la d-1\ra}\quad \text{ and }\quad b^{\la d\ra}=\binom{t+d}{d+1}+B^{\la d-1\ra}$$
and 
$$a_1^{\la d\ra}=\binom{s+d}{d+1}+(A+B)^{\la d-1\ra}\quad \text{ and }\quad b_1^{\la d\ra}=\binom{t+d}{d+1}$$
Since by induction $A^{\la d-1\ra}+B^{\la d-1\ra}\leq (A+B)^{\la d-1\ra}$, from the above equalities we get
$$a^{\la d\ra}+b^{\la d\ra}\leq a_1^{\la d\ra}+b_1^{\la d\ra}$$
This yields \eqref{three}.

\smallskip
\noindent {\bf Case 2:} $A+B=\binom{s-1+d}{d-1}+e$ for some $e\geq 0$. In this case we set
$$a_1=\binom{s+d}{d} \quad \text{ and }\quad b_1=\binom{t-1+d}{d}+e$$
We clearly have \eqref{one}. By \eqref{bin11}, we have 
$$a_1+b_1=\binom{s+d-1}{d}+\binom{s+d-1}{d-1}+b+A-\binom{s-1+d}{d-1}=a+b$$
This yields \eqref{two}. Now, since $b\leq a$, we have $t\leq s$, and so $B< \binom{s-1+d}{d-1}$. By induction we have
$$A^{\la d-1\ra}+B^{\la d-1\ra}\leq \binom{s-1+d}{d-1}^{\la d-1\ra}+e^{\la d-1\ra}$$
On the other hand, by \eqref{use00} we have
$$a_1^{\la d\ra}=\binom{s+d+1}{d+1} \quad \text{ and }\quad  b_1^{\la d\ra}=\binom{t+d}{d+1}+e^{\la d-1\ra}$$
where the second equality uses the fact that $0\leq e<\binom{t-1+d}{d-1}$. The above two displays, together with \eqref{bin11} and \eqref{use00}, yield
\begin{align*}
a^{\la d\ra}+b^{\la d\ra} & =\binom{s+d+1}{d+1}-\binom{s+d}{d}+\binom{t+d}{d+1}+A^{\la d-1\ra}+B^{\la d-1\ra} \\
&\leq a_1^{\la d\ra} -\binom{s+d}{d}+b_1^{\la d\ra} +\binom{s-1+d}{d-1}^{\la d-1\ra} \\
& =a_1^{\la d\ra}+b_1^{\la d\ra}
\end{align*}
and hence we obtain \eqref{three}. We observe that when $e=0$, we must have $A,B>0$, and so, by induction, we must have $A^{\la d-1\ra}+B^{\la d-1\ra}<\binom{s-1+d}{d-1}^{\la d-1 \ra}$. Using this, the inequality in the second line of the above display becomes strict; that is, in this case we obtain
\begin{equation}\label{strict}
a^{\la d\ra}+b^{\la d\ra}< a_1^{\la d\ra}+b_1^{\la d\ra}.
\end{equation}

Iterating this construction, obtaining $(a_{i+1},b_{i+1})$ from $(a_i,b_i)$ satisfying \eqref{one} to  \eqref{three}, one eventually finds $\ell_1\leq \ell_2$ such that $a_{\ell_1}=\binom{m-1+d}{d}$ (which implies $b_{\ell_1}=c$) and $a_{\ell_2}=a+b$ (which implies $b_{\ell_2}=0$). This, of course, shows that
$$a^{\la d\ra}+b^{\la d\ra}\leq a_{\ell_1}^{\la d\ra}+b_{\ell_1}^{\la d\ra}=\binom{m-1+d}{d}^{\la d\ra}+c^{\la d\ra}$$
 and
$$a^{\la d\ra}+b^{\la d\ra}\leq a_{\ell_2}^{\la d\ra}+b_{\ell_2}^{\la d\ra}=(a+b)^{\la d\ra}.$$
For the 'furthermore' clause note that if $c=0$ then $b_{\ell_1}=0$, and, by \eqref{strict}, we have
$$a^{\la d\ra}+b^{\la d\ra}<a_{\ell_1}^{\la d\ra}=\binom{m-1+d}{d}^{\la d\ra}.$$
\end{proof}

The above lemma yields the following (which will be used in Theorem~\ref{main}).

\begin{lemma}\label{technical}
Let $m$ and $d$ be positive integers. Suppose $a_1\leq \cdots \leq a_t$ and $b_1, \dots, b_s$ are sequences of nonnegative integers such that 
$$b_1\leq b_2= \cdots =b_s=\binom{m-1+d}{d}$$
and $b_s\geq a_i$ for all $i\leq t$. If $a_1+\cdots+a_t\leq b_1+\cdots+b_s$, then
$$a_1^{\langle d\rangle}+\cdots+a_t^{\langle d \rangle}\leq b_1^{\langle d\rangle}+\cdots+b_s^{\langle d\rangle}.$$
\end{lemma}
\begin{proof}
We proceed by induction on $(t,s)$ using the lexicographic order. By \eqref{ineq11}, if $a\leq b$ then $a^{\langle d\rangle}\leq b^{\langle d\rangle}$. The case $t=1$ follows from this observation. By Lemma~\ref{rep}, $a^{\langle d\rangle}+b^{\langle d\rangle}\leq (a+b)^{\langle d\rangle}$, and the case $s=1$ follows from this. Thus, we assume that $r,s>1$. We now consider two cases:

\medskip

\noindent \underline{Case 1.} Suppose $b_1\geq a_1$. Then the sequences $a_2\leq \cdots\leq a_{t}$ and $b_1-a_1\leq b_2\leq \cdots \leq b_s$ satisfy our hypothesis. By induction, 
 $$a_2^{\langle d\rangle}+\cdots+a_{t}^{\langle d\rangle} \leq (b_1-a_1)^{\langle d\rangle}+b_{2}^{\langle d\rangle}+\cdots + b_s^{\langle d\rangle}.$$
Using that $a_1^{\langle d\rangle}+(b_1-a_1)^{\langle d\rangle}\leq b_1^{\langle d\rangle}$, which follows from Lemma~\ref{rep}, we get the desired inequality for the original sequences.

\medskip

\noindent \underline{Case 2.} Suppose $b_1< a_1$. When $s=2$, we must have that $a_2+\cdots +a_{t}\leq b_2$, and so
$$a_1^{\langle d\rangle}+\cdots+a_t^{\langle d\rangle}\leq a_1^{\langle d\rangle} +(a_2+\cdots +a_{t})^{\langle d\rangle} \leq b_1^{\langle d\rangle}+b_2^{\langle d\rangle},$$
where the first inequality follows from part (1) of Lemma~\ref{rep} and the second from part (2). So we assume that $s>2$. If it happens that $a_1+\cdots+a_{t}\leq b_2+\cdots+b_{s}$, then we are done by induction. So we can assume that
\begin{equation}\label{cas}
a_2+\cdots+a_{t} > b_3+\cdots+b_{s}.
\end{equation}
We have that
$$b_{1} + b_2\geq a_1 +\left(a_2+\cdots +a_{t}-b_3-\cdots -b_{s}\right).$$
It follows from Lemma~\ref{rep}, using \eqref{cas}, that
$$b_{1}^{\langle d\rangle} +b_2^{\langle d\rangle} \geq a_1^{\langle d\rangle} + \left(a_2+\cdots +a_{t}-b_3-\cdots -b_{s}\right)^{\langle d\rangle}.$$
Thus, it suffices to see that 
$$b_3^{\langle d\rangle}+\cdots + b_{s}^{\langle d\rangle} + \left(a_2+\cdots +a_{t}-b_3-\cdots -b_{s}\right)^{\langle d\rangle}\geq a_2^{\langle d\rangle}+\cdots +a_{t}^{\langle d\rangle},$$
but this follows by induction.
\end{proof}

Lastly, we we will use in Theorem \ref{main} the following consequence of the above two lemmas.

\begin{corollary}\label{techcon}
Let $m$ and $d$ be positive integers. Suppose $a_1\leq \cdots \leq a_t$ and $b_1,\dots, b_s$ are sequences of positive integers such that
$$b_1=\cdots=b_{s}=\binom{m-1+d}{d}$$
and $b_s\geq a_i$ for all $i\leq t$. If $a_1+\cdots+a_t\leq b_1+\cdots+b_s$ and
\begin{equation}\label{bineq21}
a_1^{\la d\ra}+\cdots+a_t^{\la d\ra}=b_1^{\la d\ra}+\cdots+ b_s^{\la d\ra},
\end{equation}
then $s=t$ and $a_i=b_i$ for all $i$.
\end{corollary}
\begin{proof}
Lemma \ref{technical}, together with \eqref{bineq21}, imply that we must have 
\begin{equation}\label{bineq20}
a_1+\cdots+a_t= b_1+\cdots+b_s.
\end{equation}
We proceed by induction on $t$. The case $t=1$ is obvious. Now assume $t>1$. If $a_t=b_s$, then we are done by induction. Thus assume, towards a contradiction, that $a_t<b_s$. Furthermore, we may also assume that $a_t$ is the largest integer with $a_t<b_s$ and for which there is a sequence $a_1\leq \cdots\leq a_t$ satisfying \eqref{bineq21} and \eqref{bineq20}.


Just as we did in the proof of Lemma \ref{rep}, we can find $\alpha_1$ and $\alpha_t$ such that 
\begin{equation}\label{on}
0\leq \alpha_1< a_1 \leq a_t<\alpha_t\leq \binom{m-1+d}{d}
\end{equation}
\begin{equation}\label{tw}
\alpha_1+\alpha_t=a_1+a_t
\end{equation}
\begin{equation}\label{th}
a_1^{\la d\ra}+a_t^{\la d\ra}\leq \alpha_1^{\la d\ra}+\alpha_t^{\la d\ra}
\end{equation}
Hence, $\alpha_1+a_2+\dots+a_{t-1}+\alpha_t=b_1+\cdots+ b_s$, and by Lemma \ref{technical}
$$a_1^{\la d\ra}+\cdots+a_t^{\la d\ra}\leq \alpha_1^{\la d\ra}+a_2^{\la d\ra}+\cdots+a_{t-1}^{\la d\ra}+\alpha_t^{\la d\ra}\leq b_1^{\la d\ra}+\cdots+ b_s^{\la d\ra}$$
and so, by our assumption, all the inequalities become equality. We now consider two cases:

\medskip 

\noindent  {\bf Case 1:} $\alpha_1>0$. In this case, by our assumption on $a_t$ and since $\alpha_t>a_t$, we must have $s=t$ and $\alpha_1=a_2=\cdots=a_{t-1}=\alpha_t=b_s$. Thus, $a_1+a_t=b_1+b_s$. But this is impossible as $a_1\leq a_t< b_s=b_1$. So we have reached the desired contradiction.

\medskip

\noindent {\bf Case 2:} $\alpha_1=0$. In this case, by induction on $t$, we must have $s=t-1$ and $a_2=\cdots=a_{t-1}=\alpha_t=b_s$. Thus, $a_1+a_t=b_s$ and $a_1^{\la d\ra}+a_t^{\la d\ra}=b_s^{\la d\ra}$. However, since $a_1,a_t>0$, this is impossible by the furthermore clause of part (2) of Lemma~\ref{rep}.
\end{proof}

\section{An effective bound for the (standard) coefficients}\label{rough}

In this section we prove that there is a recursively computable (upper and lower) bound for all the coefficients of the Kolchin polynomial that only depends on the order of the differential system (and, of course, on the number of derivations and variables). Thus answering Question \ref{introquestion} of Kolchin's. We will use the notation set in the previous section. In particular, $(K,\D)$ is our ground differential field of characteristic zero with $m$ commuting derivations $\D=\{\d_1,\dots, \d_m\}$.

The following proposition is the key to prove the existence of the desired bound, and could be of independent interest in the general theory of differential and difference dimension polynomials~\cite{KLMP}.

\begin{proposition}\label{ontheco}
Let $E\subseteq \N^m$ and denote by $M$ the set of minimal elements of $E$ with respect to the product order $\leq$ of $\N^m$ (note that, by Dickson's lemma, $M$ is a finite set). Set $D=0$ if $M$ is empty, otherwise set
$$D=\sum_{\xi\in M}\ord\xi.$$
If we write
$$\omega_E(t)=\sum_{i=0}^m a_i\binom{t+i}{i},$$
then, for all $j=0,1,\dots,m$,
$$|a_m|+|a_{m-1}|+\cdots+|a_{m-j}|\leq D^{j}.$$
\end{proposition}
\begin{proof}
We proceed by induction on $m$ and $D$. Note that the case $m=1$, as well as the case $D=0$, clearly holds. So we now assume $m>1$ and $D>0$. 

If $a_m\neq 0$, then, by \eqref{maxi} in Section \ref{pre}, $a_m=1$ and the other coefficients vanish, so this case also holds. So we assume $a_m=0$, and we must show that, for all $j=1,\dots,m$,
$$|a_{m-1}|+|a_{m-2}|+\cdots+|a_{m-j}|\leq D^{j}.$$

As we are assuming $a_m=0$, $M$ is nonempty (by \eqref{maxi}). Let $\zeta=(v_1,\dots,v_m)\in M$. Moreoever, since $D>0$, $\zeta$ is not zero; so, without loss of generality, we may assume that $v_m\neq 0$.

Let 
$$E_1=\{(u_1,\dots,u_{m-1})\in \N^{m-1}: (u_1,\dots,u_{m-1},0)\geq \xi \text{ for some }\xi\in E\}$$
and
$$E_2=\{(u_1,\dots,u_m)\in \N^m: (u_1,\dots,u_m+1)\geq \xi \text{ for some }\xi\in E\}.$$
Clearly, 
\begin{equation}\label{polyid}
\omega_E(t)=\omega_{E_1}(t)+\omega_{E_2}(t-1).
\end{equation}
Letting $M_i$ be the minimal elements of $E_i$ for $i=1,2$, we see that $\displaystyle \sum_{\xi\in M_1}\ord\xi\leq D$. On the other hand, the tuple $\eta=(v_1,\dots,v_m-1)$ is in $M_2$, as $\zeta\in M$ and $v_m>0$, and so, since $\ord \eta<\ord \zeta$, we have $\displaystyle \sum_{\xi\in M_2}\ord\xi\leq D-1$. Thus, by induction, if we write
$$\omega_{E_1}(t)=\sum_{i=0}^{m-1}b_i\binom{t+i}{i}$$
and 
$$\omega_{E_2}(t)=\sum_{i=0}^mc_i\binom{t+i}{i}$$
we get, for all $j=1,\dots,m$,
$$|b_{m-1}|+|b_{m-2}|+\cdots +|b_{m-j}|\leq D^{j-1}$$
and
$$|c_{m}|+|c_{m-1}|+\cdots +|c_{m-j}|\leq (D-1)^{j}.$$ 

Using the identity
$$\binom{t-1+i}{i}=\binom{t+i}{i}-\binom{t+i-1}{i-1}$$
and \eqref{polyid}, we get
$$\omega_E(t)=\sum_{i=0}^{m-1}\left(b_i+c_i-c_{i+1}\right)\binom{t+i}{i}+c_m\binom{t+m}{m}$$
Now, since $a_m=0$ we get that $c_m=0$, and so the above equality yields
$$a_{m-1}=b_{m-1}+c_{m-1}-c_m\leq 1+(D-1)= D.$$
This proves the case $j=1$. So now we assume $j\geq 2$. We get
\begin{align*}
|a_{m-1}|+\cdots+|a_{m-j}| &= |b_{m-1}+c_{m-1} -c_m|+\cdots +|b_{m-j}+c_{m-j}-c_{m-j+1}| \\
&\leq  |b_{m-1}|+\cdots+|b_{m-j}|+|c_{m-1}|+\cdots +|c_{m-j}| +|c_{m-1}|+\cdots +|c_{m-j+1}|\\
&\leq D^{j-1}+(D-1)^{j}+(D-1)^{j-1} \\
&\leq D^j.
\end{align*}
Where the last inequality holds as we are assuming $j\geq 2$ and $D\geq 1$. 
\end{proof}


The above proposition, together with Fact~\ref{factKol}, essentially shows that, in order to produce the desired bound for the coefficients of the Kolchin polynomial, all that is missing is an upper bound for the order of a characteristic set of a prime differential ideal that depends solely on the order of the differential system. Such a bound was obtained in \cite{GO}, its recursive formula uses the Ackermann function. We recall that this (nonprimitive recursive) function is defined as $A:\mathbb N\times\mathbb N\to \mathbb N$ 
$$
A(x,y) = \begin{cases} y + 1 & \text{ if } x = 0 \\ A(x-1,1) & \text{ if } x > 0 \text{ and } y = 0 \\ A(x-1,A(x,y-1)) & \text{ if } x,y > 0. \end{cases}
$$
From the Ackermann function we define $C_{r,m}^n$, for $r\geq 0$ and $m,n>0$, recursively as follows: 
$$C_{0,m}^1=0, \quad\; C_{r,m}^1=A(m-1,C_{r-1,m}^1), \quad \text{ and } \quad C_{r,m}^n=C_{C_{r,m}^{n-1},m}^1.$$
For example, a straigthforward computation yields
$$C_{r,1}^n=r, \quad\; C_{r,2}^n=2^n r \quad \text{ and }\quad C_{r,3}^1=3(2^r-1).$$

From \cite[Proposition 6.1]{GO} we have

\begin{fact}\label{boundchar}
Let $\Sigma\subset K\{x_1,\dots,x_n\}$ be of order at most $r$ and $\P$ a prime component. Then a characteristic set for $\P$ has order at most $C_{r,m}^n$.
\end{fact}

Consequently, we get

\begin{corollary}\label{mainrough}
Let $\Sigma\subset K\{x_1,\dots,x_n\}$ be of order at most $r$ and $\P$ a prime component. If 
$$\omega_\P(t)=\sum_{i=0}^ma_i\binom{t+i}{i}$$ 
is the Kolchin polynomial of $\P$, then, for all $j=0,\dots,m$,
$$|a_m|+|a_{m-1}|+\cdots+|a_{m-j}|\leq n D^{j}$$
where 
$$D=\binom{C_{r,m}^n+m-1}{C_{r,m}^n}C_{r,m}^n.$$
\end{corollary}

\begin{proof}
Recall that $\omega_\P=\sum_{i=1}^n\omega_{E_i}$ where $E_i$ denotes the set of $\xi\in \N^m$ such that $\d^\xi x_i$ is a leader of a characteristic set of $\P$. By Fact \ref{boundchar}, the elements of the set $M_i$ of minimal elements of $E_i$ have order at most $C_{r,m}^n$. Since the number of $m$-tuples of order $s\in\N$ is $\binom{s+m-1}{s}$, we get that the number of elements in $M_i$ is at most 
$$\displaystyle \binom{C_{r,m}^n+m-1}{C_{r,m}^n},$$
and so
$$\sum_{\xi\in M_i}\ord \xi\leq \binom{C_{r,m}^n+m-1}{C_{r,m}^n}C_{r,m}^n.$$
The result now follows immediately from Proposition~\ref{ontheco}.
\end{proof}

\begin{remark}\label{ritt}
In the case when the differential type of the prime component $\P$ is zero, it is possible to obtain the better bound $a_0\leq n\cdot(C_{r,m}^n)^m$. This upper bound appears in \cite[Proposition 6.2]{GO}. On the other hand, if we specialize the bound in the above corollary to $m=1$, we obtain $a_0\leq nr$, and so we recover the bound found by Ritt in \cite[Chapter 6]{Ritt}.
\end{remark}

We now have an upper and lower bound for the coefficients of the Kolchin polynomial; in particular, this yields an upper bound for the typical differential dimension that only depends on the order of the differential system (and on the fixed data $m$ and $n$). This already answers the question of Kolchin (Question \ref{introquestion}). On the other hand, note that in the case $m=2$ and $n=1$, if the differential type of $\P$ is zero, the above corollary says that the typical differential dimension of $\P$ satisfies
$$a_0\leq 4r^2(2r+1)^2.$$
We saw in Example \ref{introex} that in this case $a_0$ can be $r^2$. But can be it be larger? We will see, in Corollary \ref{final2}, that in fact $r^2$ is an upper bound for $a_0$ (and so we obtain an optimal bound in this situation). To prove this, we will need to prove further (nontrivial) combinatorial properties of the leaders of characteristic sets of prime differential ideals. This is done in the next two sections.

\section{A combinatorial-structural result for characteristic sets}\label{combi}

In this section, and the next, we study combinatorial properties of characteristic sets of prime differential ideals. Our main goal is to produce better bounds than the ones obtained in Corollary \ref{mainrough} for the typical differential dimension (see the final paragraph of Section \ref{rough}).

We keep the same conventions and notation of previous sections. Henceforth we let $\n=\{1,\dots,n\}$. We will consider two different orders $\leq$ and $\unlhd$ on $\Nn$. Given two elements $(\xi,i)$ and $(\eta,j)$ of $\Nn$, we set $(\xi,i)\leq (\eta,j)$ if and only if $i=j$ and $\xi\leq \eta$ with respect to the product order of $\N^m$. On the other hand, if $\xi=(u_1,\dots,u_m)$ and $\eta=(v_1,\dots,v_m)$, we set $(\xi,i) \unlhd (\eta,j)$ if and only if 
$$
(\ord \xi,i,u_1,\dots,u_{m}) \leq_{\text{lex}} (\ord \eta,j,v_1,\dots,v_m)
$$ 
where $\leq_{\text{lex}}$ denotes the (left) lexicographic order.  Note that if $x=(x_1,\ldots,x_{n})$ are differential indeterminates and we identify $(\xi,i)$ with $\d^\xi x_i:=\d_1^{u_1}\cdots\d_m^{u_m}x_i$, then $\leq$ induces an order on the set of algebraic indeterminates $\{\d^\xi x_i:(\xi,i)\in\Nn\}$ given by $\d^\xi x_i\leq \d^\eta x_j$ if and only if $\d^\eta x_j$ is a derivative of $\d^\xi x_i$. On the other hand, the ordering $\unlhd$ induces the canonical orderly ranking on the set of algebraic indeterminates.

The goal of this section is to prove the following combinatorial-structural theorem on the leaders of characteristic sets of prime differential ideals. Recall that for $\xi$ and $\eta$ in $\N^m$, we let $\lUB(\xi,\eta)$ be the least upper bound of $\xi$ and $\eta$ with respect to the product order.

\begin{theorem}\label{combinatorialchar}
Suppose $\Sigma\subset K\{x_1,\dots,x_n\}$ has order at most $r$ and $\P$ is a prime component of $\Sigma$. Suppose further that the order of a characteristic set $\Lambda$ of $\P$ is $H>r$, and let $E_j$ be the set of $\xi\in \N^m$ such that $\d^\xi x_j$ is a leader of $\Lambda$. Then, for each $r\leq h\leq H$, there is $1\leq j\leq n$ and distinct $\xi,\xi'\in E_j$ of order at most $h$ with the following property
\begin{enumerate}
\item [($\dagger$)] for every sequence 
$$\xi=\eta_1,\; \eta_2,\; \dots,\; \eta_s=\xi'$$ 
of $E_j$ with each $\eta_i<\lUB(\xi,\xi')$ and of order at most $h$, there is $1\leq i\leq s$ such that
$$\ord\lUB(\eta_i,\eta_{i+1})>h.$$
\end{enumerate}
Furthermore, if $r\leq h<H$ and $\Lambda$ has no element of order $h+1$, then we can find $1\leq j\leq n$ and distinct $\xi,\xi'\in E_j$ of order at most $h$ with the following property
\begin{enumerate}
\item [($\dagger'$)] same as $(\dagger)$ except that the last inequality now becomes
$$\ord\lUB(\eta_i,\eta_{i+1})>h+1.$$
\end{enumerate}
\end{theorem}

\begin{remark}
The reason we call this a ``combinatorial-structural'' theorem is the following: Suppose $n=1$ and write $\Lambda=\{f_1,\dots,f_s\}$ with $\ord f_1\leq\cdots\leq \ord f_s$. Then the theorem says that $\ord f_1$ and $\ord f_2$ are at most $r$, and $\ord f_3\leq \ord\lUB(\xi_{f_1},\xi_{f_2})$ where $\xi_{f_i}$ is such that $\d^{\xi_{f_i}}x$ is the leader of $f_i$. In other words, the leader of $f_3$ cannot be that ``far'' from the leaders of the previous two elements. Similar observations can be made for the rest of the $f_i$'s. For instance,
$$\ord f_4\leq \max\{\ord\lUB(\xi_{f_1},\xi_{f_2}), \ord\lUB(\xi_{f_1},\xi_{f_3}), \ord\lUB(\xi_{f_2},\xi_{f_3})\}.$$
In particular, for $i\geq 2$, we obtain a bound on $\ord f_i$ in terms of the leaders of the previous elements.
\end{remark}

The proof of the theorem makes use of the theory of differential kernels from \cite{GO}. Let us review the results that we will need. For $\alpha=(\xi,i)\in \Nn$, we set $\ord\alpha=\ord\xi$. Also, for each nonnegative integer $r$ we let 
$$\Gamma(r)=\{\alpha\in \Nn: \ord \alpha\leq r\}.$$

\begin{definition}
A differential kernel of length $r\in \N$ over $K$ is a field extension of the form
$$L=K(a_i^\xi:(\xi,i)\in \Gamma(r))$$
such that there exists derivations
$$D_k:K(a_i^\xi:(\xi,i)\in \Gamma(r-1))\to L,\; \text{ for }k=1,\dots,m,$$
extending $\d_k$ and $D_ka_i^\xi=a_i^{\xi+{\bf k}}$ where $\bf k$ denotes the $m$-tuple whose $k$-th entry is one and zeroes elsewhere.
\end{definition}


Given a field extension $L$ of $K$ of the form
\begin{equation}\label{kernelform}
L=K(a_i^\xi:(\xi,i)\in \Gamma(r)),
\end{equation}
we say that $(\xi,i)\in \Nn$ is a \emph{leader} of $L$ if $a_i^{\xi}$ is algebraic over $K(a_j^{\eta}: (\eta,j)\lhd (\xi,i))$. A leader $(\xi,i)$ is said to be \emph{minimal} if there is no leader $(\eta,i)$ with $\eta<\xi$. Note that the notions of leader and minimal leader make sense even when we allow $r=\infty$ in \eqref{kernelform}, where we set $\Gamma(\infty)=\Nn$. 

\begin{definition}
An $n$-tuple $g=(g_1,\dots,g_n)$ contained in a differential field extension $(M,\{\dd_1,\dots,\dd_m\})$ of $(K,\D)$ is said to be a regular realization of the differential kernel $L=K(a_i^\xi:(\xi,i)\in \Gamma(r))$ if the tuple
$$(\dd^\xi g_i:(\xi,i)\in \Gamma(r))$$
is a generic specialization of $(a_i^\xi:(\xi,i)\in \Gamma(r))$ over $K$ (in the algebraic sense). The tuple $g$ is said to be a principal realization of $L$ if all the minimal leaders of the differential field 
$$K\langle g\rangle=K(\dd^\xi g_i:(\xi,i)\in\Nn)$$ 
have order at most $r$.
\end{definition}

\begin{remark}\label{special}
If $f$ is a principal realization and $g$ is a regular realization of the differential kernel $L$, then $g$ is a \emph{differential} specialization of $f$ over $K$. This is the content of \cite[Lemma 2.7]{GO}. 
\end{remark}

From \cite[Theorem 3.1]{GO}, we get

\begin{fact}\label{principal}
Let $L=K(a_i^\xi:(\xi,i)\in \Gamma(r))$ be a differential kernel over $K$. Suppose that for every $1\leq j\leq n$ the following condition holds
\begin{enumerate}
\item [($\sharp$)] for every pair of distinct minimal leaders of $L$ of the form $(\xi,j)$ and $(\xi',j)$ there exists a sequence
$$(\xi,j)=(\eta_1,j),\; (\eta_2,j),\;\cdots,\; (\eta_s,j)=(\xi',j)$$
of minimal leaders such that $\eta_i<\lUB(\xi,\xi')$ and $\ord\lUB(\eta_i,\eta_{i+1})\leq r$ for $i=1,\dots,s-1$.
\end{enumerate}
Then, the differential kernel $L$ has a principal realization. 
\end{fact}

We can now prove Theorem \ref{combinatorialchar}.

\begin{proof}[Proof of Theorem \ref{combinatorialchar}]
Suppose, towards a contradiction, that there is $r\leq h\leq H$ such that for all $j$ and distinct $\xi,\xi'\in E_j$, of order at most $h$, there is a sequence as in ($\dagger$) but with $\ord\lUB(\eta_i,\eta_{i+1})\leq h$ for all $i=1,\dots,s$. Note that the latter implies that $\Lambda$ has no elements of order $h$. We will show that then $\Lambda$ has order at most $h-1$, contradicting the choice of $H$. 

Let $a=(a_1,\dots,a_n)$ be a differential generic point over $K$ of $\P$ in a universal differential field extension $(\mathcal U,\D)$ of $(K,\D)$. Set $a_i^\xi:=\d^\xi a_i$, and consider the differential kernel $L=K(a_i^\xi:(\xi,i)\in \Gamma(h))$. By Fact~\ref{principal}, and the assumption in the above paragraph, $L$ has a principal realization. Since $\mathcal U$ is universal, we may assume that this principal realization of $L$ is witnessed inside of $\mathcal U$; that is, there is tuple $b=(b_1,\dots,b_n)$ from $\mathcal U$ such that $(\d^\xi b_i:(\xi,i)\in \Gamma(h))$ is a generic specialization of $(\d^\xi a_i:(\xi,i)\in \Gamma(h))$ over $K$, and all minimal leaders of the differential field $K\langle b\rangle$ have order at most $h$. 

Since $\Sigma$ has order at most $r\leq h$ and $a$ is a zero of $\Sigma$, we get that $b$ is also a zero of $\Sigma$. Moreover, by Remark \ref{special}, $a$ is a differential specialization of $b$ over $K$. But $a$ is a differential generic point of the prime component $\P$ of $\Sigma$, so $b$ must also be a differential generic point of $\P$ over $K$. Hence, because all minimal leaders of $K\langle b\rangle$ have order at most $h$, the characteristic set $\Lambda$ of $\P$ has order at most $h$. Moreover, as we observed in the first paragraph of the proof, $\Lambda$ does not have elements of order $h$. Thus, $\Lambda$ has order at most $h-1$, and we have reached the desired contradiction.

For the ``furthermore'' clause, suppose, towards a contradiction, that there is $r\leq h<H$ such that $\Lambda$ has no elements of order $h+1$ and that for all $j$ and distinct $\xi,\xi'\in E_j$, of order at most $h$, there is a sequence as in ($\dagger$) but with $\ord\lUB(\eta_i,\eta_{i+1})\leq h+1$ for all $i=1,\dots,s$. The same argument as above, but now setting $L$ to be the differential kernel $K(a_i^\xi:(\xi,i)\in \Gamma(h+1))$, yields that $\Lambda$ has order at most $h+1$. But as we are assuming $\Lambda$ has no elements of order $h+1$, we get that $\Lambda$ has order at most $h$, again contradicting the choice of $H$. 
\end{proof}

\section{Main results on estimates for typical differential dimension}\label{several}

In this section, and the next, we give a (recursive) construction of a bound, denoted $\M(r,m,n,\tau)$, for the typical differential dimension of any prime component $\P$ of a given differential system, which only depends on the order $r$ of the system, the differential type $\tau$ of $\P$, and the number of derivations $m$ and variables $n$. This bound, for the differential type, is significantly better than the one produced by Corollary \ref{mainrough} (c.f. Section \ref{compu} below).

We carry forward the notation and conventions from previous sections. In particular, $r$ will be a nonnegative integer (usually denoting the order of the differential system), $m$ is the number of derivations and $n$ the number of differential variables. Also, $(K,\D=\{\d_1,\dots,\d_m\})$ is our ground differential field of characteristic zero. At this point it might be convenient for the reader to recall some of the terminology presented in Section \ref{pre}, such as compressed sets, Hilbert-Samuel function, Macaulay's and Gotzmann's theorems, etc.

In order to prove the main result of this section we will make use of a special sequence. Let $\bar\mu$ be the sequence $\mu_1,\mu_2,\dots,\mu_L$ of $\N^m$ defined recursively as follows: Let $\mu_1=(r,0,\dots,0)$ and for $i\geq 2$, as long as it is possible,
$$\mu_{i}=\max_{\unlhd}\{\xi\in \N^m: \ord\xi=r+i-2 \text{ and } \xi\not\geq \mu_1,\dots,\mu_{i-1}\}.$$
For instance, if $m\geq 3$, $\mu_2=(r-1,1,0\dots,0)$, $\mu_3=(r-1,0,2,0\dots,0)$, etc. 

The sequence $\bar \mu$ can be more explicitly constructed as follows; for $i\geq 2$,
\begin{enumerate}
\item[(i)] if $\mu_i = (u_1,\dots,u_s,0,\dots,0,u_m)$ with $s < m-1$ and $u_s > 0$, then
$$
\mu_{i+1} = (u_1,\dots,u_s-1, u_m+2,0,\dots,0)
$$
\item[(ii)] if $\mu_i = (u_1,\dots,u_{m-1},u_m)$ with $u_{m-1} > 0$, then
$$
\mu_{i+1} = (u_1,\dots,u_{m-1} - 1, u_m+2).
$$
\end{enumerate}

For example, when $m=2$, the sequence $\bar\mu$ is given by
$$
\mu_1=(r,0), \; \mu_2=(r-1,1),\; \mu_3=(r-2,3),\; \mu_4=(r-3,5),\dots,\; \mu_{r+1}=(0,2r-1),
$$
and so in this case the length of the sequence is $L=r+1$.

In \cite[\S3.2]{LO}, it is observed that $\bar\mu$ is a compressed subset of $\N^m$. We also have that the order of the last element of the sequence satisfies $\ord\mu_L+1=C_{r,m}^1$, where $C_{r,m}^n$ was defined in Section \ref{rough} (in terms of the Ackermann function). Furthermore, it is easy to see that $H_{\bar\mu}(\ord\mu_L)=0$ where $H_{\bar\mu}$ denotes the Hilbert-Samuel function of $\bar\mu$. This latter fact implies that $\bar\mu$ has \emph{finite volume}; that is, 
$$\Vol\bar\mu:=|\{\xi\in \N^m: \xi\not\geq \mu_i \text{ for all } 1\leq i\leq L\}|$$
is finite. 

\begin{remark}\label{vol}
From the construction of $\bar\mu$, one can check that $\Vol\bar\mu$ is equal to the sum of the $m$-th entry of the tuples in the sequence $\bar\mu$.
\end{remark}

For $1\leq\ell\leq L$, let $\bar\mu|_\ell$ denote the subsequence $\mu_1,\dots,\mu_\ell$.

\begin{proposition}\label{shape}
For $1\leq \ell\leq L$ we have
$$\omega_{\bar\mu|_\ell}(t)=\sum_{i=1}^{m-1}c_{m-i}\binom{t+i}{i}+c$$
where $\mu_\ell=(c_1,\dots,c_m)$ and $c$ is the sum of the $m$-th entry of the tuples in $\bar\mu|_\ell$. In particular, $\omega_{\bar\mu}(t)$ is constant, equal to $\Vol\bar\mu$.
\end{proposition}

\begin{proof}
Note that this is trivially true when $m=1$, so we assume $m\geq 2$. We proceed by induction on $1\leq \ell\leq L$. The case $\ell=1$ is clear, as $\bar\mu|_1=\mu_1=(r,0,\dots,0)$ and so 
$$\omega_{\mu|_1}(t)=r\binom{t+m-1}{m-1}.$$
For $\ell=2$, $\bar\mu|_2$ consists of $(r,0,\dots,0)$ and $(r-1,1,0,\dots,0)$. Note that a point $(v_1,\dots,v_m)\in \N^m$ is greater than or equal to $\mu_2$ but not to $\mu_1$ (with respect to the product order) if and only if $v_1=r-1$ and $v_2\geq 1$. Thus
$$\omega_{\bar\mu|_2}(t)=(r-1)\binom{t+m-1}{m-1}+\binom{t+m-2}{m-2}$$
Now assume $\ell\geq 3$. We consider the two possible shapes that $\mu_\ell$ can take according to the construction of $\bar \mu$:

\medskip

\noindent \underline{Case 1.} Suppose $\mu_{\ell-1}=(u_1,\dots,u_s,0,\dots,0,u_m)$ with $s<m-1$ and $u_s>0$. Then, by construction of $\bar\mu$, 
$$\mu_{\ell}=(u_1,\dots,u_s-1,u_m+2,0,\dots,0).$$
A point $(v_1,\dots,v_m)\in \N^m$ is greater than or equal to $\mu_\ell$ but not to any element of $\bar\mu|_{\ell-1}$ if and only if
$$v_1=u_1,\;\dots,\;v_{s}=u_s-1 \;\text{ and }\; v_{s+1}\geq u_m+2.$$
By induction 
$$\omega_{\bar\mu|_{\ell-1}}(t)=\sum_{i=m-s}^{m-1} u_{m-i}\binom{t+i}{i}+c,$$
hence the above yields
\begin{align*}
\omega_{\bar\mu|_\ell}(t)
&=\omega_{\bar\mu|_{\ell-1}}-\binom{t+m-s}{m-s}+(u_{m}+2)\binom{t+m-s-1}{m-s-1} \\
&= \sum_{i=m-s+1}^{m-1}u_{m-i}\binom{t+i}{i}+(u_{s}-1)\binom{t+m-s}{m-s}+(u_m+2)\binom{t+m-s-1}{m-s-1}+c
\end{align*}
This is the desired shape of $\omega_{\bar\mu|_\ell}$, noting that, by induction, $c$ is the sum of the $m$-th entry of the tuples in $\bar\mu|_\ell$.

\medskip

\noindent \underline{Case 2.} Suppose $\mu_{\ell-1}=(u_1,\dots,u_{m-1},u_m)$ with $u_{m-1}>0$. Then, by construction of $\bar\mu$,
$$
\mu_{\ell}=(u_1,\dots,u_{m-1}-1,u_m+2).
$$
A point $(v_1,\dots,v_m)\in \N^m$ is greater than or equal to $\mu_\ell$ but not to any element of $\bar\mu|_{\ell-1}$ if and only if
$$v_1=u_1,\;\dots,\;v_{m-1}=u_{m-1}-1 \;\text{ and }\; v_{m}\geq u_m+2.$$
By induction 
$$\omega_{\bar\mu|_{\ell-1}}(t)=\sum_{i=1}^{m-1} u_{m-i}\binom{t+i}{i}+c,$$
where $c$ is the sum of the $m$-th entries of the tuples in $\bar\mu|_{\ell-1}$. The above yields
\begin{align*}
\omega_{\bar\mu|_\ell}(t)
&=\omega_{\bar\mu|_{\ell-1}}-\binom{t+1}{1}+(u_{m}+2) \\
&= \sum_{i=2}^{m-1}u_{m-i}\binom{t+i}{i}+(u_{m-1}-1)\binom{t+1}{1}+(c+u_m+2)
\end{align*}
This yields the desired shape of $\omega_{\bar\mu|_\ell}$. 

The ``in particular'' clause follows from the fact that $\mu_L$ is of the form $(0,\dots,0,u_m)$ and from Remark~\ref{vol}.
\end{proof}

For our purposes we will need a more general sequence. Let $\bar \mu^{(1)}$ be the sequence constructed as above, starting at $(r,0,\dots,0)$, inside of $\N^m\times\{n\}$ (i.e., the $n$-th copy of $\N^m$ in $\Nn$). Thus, $\bar\mu^{(1)}$ is of the form 
$$
((\mu^{(1)}_1,n),\dots,(\mu^{(1)}_{L_1},n)).
$$ 
Similarly, let $\bar \mu^{(2)}$ be the sequence constructed as above but replacing $r$ with $C_{r,m}^1$, that is, now starting at $(C_{r,m}^1,0,\dots,0)$, inside of $\N^m\times\{n-1\}$. Then, $\bar\mu^{(2)}$ has the form 
$$
\bar\mu^{(2)}=((\mu^{(2)}_1,n-1),\dots,(\mu^{(2)}_{L_2},n-1)).
$$
Continuing in this fashion, we build $\bar \mu^{(j)}$ for $j=3,\dots n$ as the sequence constructed as above but with $C_{r,m}^{j-1}$ instead of $r$, that is, starting at $(C_{r,m}^{j-1},0,\dots,0)$, inside of $\N^m\times\{n-j+1\}$. 

Note that, for each $j=1,\dots,n$, the order of the last element of $\bar\mu^{(j)}$ satisfies $\ord\mu^{(j)}_{L_j}+1=C_{r,m}^{j}$. We now define a new sequence $\bar \mu$ build as the concatenation of $\bar \mu^{(1)},\bar\mu^{(2)},\dots,\bar \mu^{(n)}$. The length of this sequence, denote by $L$, is $L_1+\dots+L_n$ where $L_j$ is the length of $\bar\mu^{(j)}$.

\begin{definition}
By construction, for every $0\leq \tau< m$, there is a unique element of the sequence $\bar\mu^{(n)}\subseteq \N^m\times\{1\}$ of the form
$$((0,\dots,0,u_{m-\tau},0\dots,0),1)$$
where $u_{m-\tau}\neq 0$ is in the $(m-\tau)$-th entry. We let $\M(r,m,n,\tau):=u_{m-\tau}$. Also, we denote by $\l_\tau$ the unique integer $1\leq \l_\tau\leq L_n$ such that 
$$\mu^{(n)}_{\l_\tau}=((0,\dots,0,\M(r,m,n,\tau),0\dots,0),1).$$ 
It is important to point out that this is the first element in the sequence $\bar\mu^{(n)}$ with the property that the first $(m-\tau-1)$ entries are all zero.
\end{definition}

We can now state and prove our main theorem. Recall that, for any $1\leq \ell\leq L_n$, we let $\bar\mu^{(n)}|_{\ell}$ denote the (sub)sequence $\mu^{(n)}_1,\dots,\mu^{(n)}_\ell$. Similarly, for $1\leq \ell\leq L$, we let $\bar\mu|_\ell$ be the subsequence obtained by truncating $\bar\mu$ at level $\ell$.

\begin{theorem}\label{main}
Let $\Sigma\subset K\{x_1,\dots,x_n\}$ be of order at most $r$ and $\P$ a prime component. If $\P$ has differential type $\tau<m$, then, for all $s$,
\begin{equation}\label{tired}
\omega_\P(s)\leq \omega_{\bar\mu^{(n)}|_{\l_\tau}}(s)+\omega_{\bar\mu^{(n-1)}}(s)+\cdots+\omega_{\bar\mu^{(1)}}(s)
\end{equation}
\end{theorem}

\begin{remark}\label{simple} It is easy to see that the theorem holds when $m=1$ or $r=0$.
\begin{enumerate}
\item In the case when $m=1$, we have $\tau=0$ and $\bar\mu^{(j)}=(r,n-j+1)$ for all $j$. Also, $\bar\mu^{(n)}|_{\l_\tau}=(r,1)$, and so the right-hand-side of \eqref{tired} becomes $nr$. Hence, in this case, inequality \eqref{tired} holds by Remark \ref{ritt}.
\item In the case when $r=0$, we have $\bar\mu^{(j)}=((0,\dots,0),n-j+1)$ for all $j$. So the right-hand-side of \eqref{tired} becomes zero. On the other hand, Fact \ref{boundchar} tells us that a characteristic set of $\P$ must have order zero, and so, as $\tau<m$, we have $\omega_\P=0$.
\end{enumerate}
\end{remark}

\begin{proof}[Proof of Theorem \ref{main}]
By Remark \ref{simple}, we assume that $m>1$ and $r>0$. Let $E_j$ be the set of $\xi\in \N^m$ such that $\d^\xi x_j$ is a leader of a characteristic set $\L$ of $\P$. Recall that by Fact \ref{factKol} we have $\omega_\P=\sum_{j=1}^n\omega_{E_j}$. Let $E=E_1\times\{n\}\cup\cdots\cup E_n\times\{1\}\subseteq \Nn$, and define the function $H_E:\N\to\N$ by $H_E(d)=\sum_{j=1}^nH_{E_j}(d)$, where $H_{E_j}$ is the Hilbert-Samuel function of $E_j$. Similarly, let $\ell$ be the length of the (concatenated) sequence $\bar\mu^{(1)},\bar\mu^{(2)},\dots,\bar\mu^{(n)}|_{\l_\tau}$, and set $H_{\bar\mu|_\ell}:\N\to\N$ to be the function defined by 
$$H_{\bar\mu|_\ell}(d)=H_{\bar\mu^{(1)}}(d)+\cdots+H_{\bar\mu^{(n-1)}}(d)+H_{\bar\mu^{(n)}|_{\l_\tau}}(d).$$

We first make a few observations on $H_{\bar\mu|_\ell}$, and explain how Lemma~\ref{technical} and Corollary~\ref{techcon} will be applied. Note that if $0\leq d<\ord\mu^{(1)}_1=r$ then all the $H_{\bar\mu^{(j)}}(d)$'s, as well as $H_{\bar\mu^{(n)}}|_{\l_\tau}(d)$, equal $\displaystyle \binom{m-1+d}{d}$, the number of $m$-tuples of order $d$. Furthermore, when $\ord\mu^{(j)}_{L_j}\leq d\leq \ord\mu^{(j+1)}_{L_{j+1}}$, we get
$$0=H_{\bar\mu^{(1)}}(d)=\cdots=H_{\bar\mu^{(j)}}(d)\leq H_{\bar\mu^{(j+1)}}(d)\leq H_{\bar\mu^{(j+2)}}(d)=\cdots=H_{\bar\mu^{(n)}}(d)=\binom{m-1+d}{d}$$
and when $d=\ord\mu^{(j)}_{L_j}$ the second inequality becomes equality. Thus, where appropriate, we will apply Lemma~\ref{technical} and Corollary~\ref{techcon} with the $H_{\bar\mu^{(j)}}(d)$'s in place of the $b_j$'s, and the $H_{E_j}(d)$'s in place of the $a_j$'s, for a fixed $d$. For instance, if $H_E(d)\leq H_{\bar\mu|_\ell}(d)$ then Lemma~\ref{technical} implies that:
$$H_{E_1}(d)^{\la d\ra}+\cdots+H_{E_n}(d)^{\la d\ra}\leq H_{\bar\mu^{(1)}}(d)^{\la d\ra}+\cdots+H_{\bar\mu^{(n-1)}}(d)^{\la d\ra}+H_{\bar\mu^{(n)}|_{\ell_\tau}}(d)^{\la d\ra}$$
This fact will be used repeatedly throughout the proof.

On the other hand, we can compute $H_{\bar\mu|_{\ell}}(d)$ in terms of $H_{\bar\mu|_{\ell}}(d-1)$ as follows: when $d\neq \ord\mu_{1}^{(i)}$ and $d\leq \ord\mu^{(n)}_{\l_\tau}$ we get, by Macaulay's theorem (see Theorem \ref{Macaulay}) and the fact that each $\bar\mu^{(j)}$ is compressed, 
\begin{equation}\label{wish}
H_{\bar\mu|_\ell}(d)=H_{\bar\mu^{(n)}|_{\ell_\tau}}(d-1)^{\la d-1\ra} + \sum_{j=1}^{n-1}H_{\bar\mu^{(j)}}(d-1)^{\la d-1\ra}-1,
\end{equation}
and when $d= \ord\mu_{1}^{(i)}$ for some $i$, as $m\geq 2$, we get
\begin{equation}\label{notwish}
H_{\bar\mu|_\ell}(d)=H_{\bar\mu^{(n)}|_{\ell_\tau}}(d-1)^{\la d-1\ra} + \sum_{j=1}^{n-1}H_{\bar\mu^{(j)}}(d-1)^{\la d-1\ra}-2,
\end{equation}
finally, when $d> \ord\mu^{(n)}|_{\l_\tau}$ we get
\begin{equation}\label{finalwish}
H_{\bar\mu|_\ell}(d)=H_{\bar\mu^{(n)}|_{\ell_\tau}}(d-1)^{\la d-1\ra} + \sum_{j=1}^{n-1}H_{\bar\mu^{(j)}}(d-1)^{\la d-1\ra}\end{equation}

We now go back to the proof. It suffices to show that for all $d\geq 0$
\begin{equation}\label{mainineq}
H_E(d)\leq H_{\bar\mu|_\ell}(d)
\end{equation}
Indeed, if this is the case, for all $s$ we get
$$\omega_\P(s)=\sum_{j=1}^n\omega_{E_j}(s)=\sum_{d=0}^sH_E(d)\leq \sum_{d=0}^sH_{\bar\mu|_\ell}(d)=\omega_{\bar\mu^{(n)}|_{\l_\tau}}(s)+\omega_{\bar\mu^{(n-1)}}(s)+\cdots+\omega_{\bar\mu^{(1)}}(s).$$

Let $H$ be the order of $\Lambda$ (recall this is a characteristic set of $\P$). To prove \eqref{mainineq}, we proceed by induction on $d$, we consider fours cases: $0\leq d\leq r$, $r< d\leq H$, $H< d\leq \ord\mu^{(n)}_{\l_\tau}$, and finally $d>\ord\mu^{(n)}_{\l_\tau}$. 

\medskip

\noindent \underline{Case 1: $0\leq d\leq r$.} When $0\leq d<r$, there are no elements of $\bar\mu$ of order less than or equal to $d$, so 
$$H_{\bar\mu|_{\ell}}(d)=n\cdot\binom{m-1+d}{d},$$
where the binomial expression gives the number of $m$-tuples of $\N^m$ of order $d$. Hence, in this case, we clearly have $H_E(d)\leq H_{\bar\mu|_{\ell}}(d)$. Now assume $d=r$. We consider two cases: when $|E|=1$ and when $|E|>1$. First, suppose $|E|=1$. Since $\tau<m$, in this case we must have $n=1$, and so we can write $E=\{\xi\}$ for some $\xi\in \N^m$. Moreover, by Theorem~\ref{combinatorialchar}, $\ord\xi\leq r$. We may assume $\ord\xi>0$ (otherwise \eqref{mainineq} is obviously true). This implies that $\tau=m-1$, and so $\l_\tau=1$, which in turn implies $\bar\mu_{\ell}=\bar\mu_{\l_\tau}=\{(r,0,\dots,0)\}$. This shows that
$$H_{\bar\mu|_{\ell}}(r)=\binom{m-1+r}{r}-1,$$
and so $H_E(r)\leq H_{\bar\mu|_{\ell}}(r)$. Now assume $|E|\geq 2$. By Theorem \ref{combinatorialchar}, $E$ must have two elements of order at most $r$. On the other hand, $\bar\mu|_\ell$ has no elements of order less than $r$ and exactly two of order $r$. It follows that $H_E(r)\leq H_{\bar\mu|_{\ell}}(r)$.

\medskip

\noindent \underline{Case 2: $r< d\leq H$.} In this case we can apply the ``furthermore'' clause of Theorem~\ref{combinatorialchar} (with $d-1$ in place of $h$). So, either $E$ has an element of order $d$, or there is $j_1$ and distinct $\xi,\xi'\in E_{j_1}$ of order at most $d-1$ satisfying property ($\dagger'$) of that theorem. But now, recalling that we are assuming $r>0$ and so $d>1$, Theorem~\ref{HS} implies that either $E$ has an element of order $d$ or
$$H_{E_{j_1}}(d)<H_{E_{j_1}}(d-1)^{\la d-1\ra}.$$
In both cases we get
\begin{equation}\label{anycase}
H_E(d)\leq \sum_{j=1}^nH_{E_{j}}(d-1)^{\la d-1\ra}-1
\end{equation}
We now consider two cases. For the first case, suppose $d\neq \ord\mu^{(i)}_{1}$ for $i=1,\dots,n$. In this case we have \eqref{wish}. Induction, together with Macaulay's theorem and Lemma~\ref{technical}, yield
$$H_E(d)\leq \sum_{j=1}^nH_{E_j}(d-1)^{\la d-1\ra}-1\leq H_{\bar\mu|_{\ell}}(d),$$
where the first inequality uses \eqref{anycase}. This yields the desired inequality.
For the second case, suppose $d=\ord\mu^{(i)}_{1}$ for some $i$. We now have \eqref{notwish}. Towards a contradiction, assume $H_E(d)>H_{\bar\mu|_\ell}(d)$. We then have, again by induction, Macaulay's theorem and Lemma~\ref{technical}, 
$$H_E(d)\leq \sum_{j=1}^nH_{E_j}(d-1)^{\la d-1\ra}-1\leq H_{\bar\mu_\ell}(d)+1\leq H_E(d).$$
Thus all of the above are equalities, and so
$$\sum_{j=1}^nH_{E_j}(d-1)^{\la d-1\ra} = H_{\bar\mu^{(n)}|_{\l_\tau}}(d-1)^{\la d-1\ra}+\sum_{j=1}^{n-1}H_{\bar\mu^{(j)}}(d-1)^{\la d-1\ra}.$$ 
By Corollary \ref{techcon}, after possibly reordering the $E_j$'s (or the variables $x_j$'s rather), we get $H_{E_j}(d-1)=H_{\bar\mu^{(j)}}(d-1)$, for all $j=1\dots,n-1$, and $H_{E_n}(d-1)=H_{\bar\mu^{(n)}|_{\l_\tau}}(d-1)$. In other words, 
$$H_{E_j}(d-1)=0\; \text{ for }j<i, \; \text{ and } \; H_E^{(j)}(d-1)=\binom{m-2+d}{d-1} \; \text{ for }j\geq i.$$
Now, by Theorem \ref{combinatorialchar}, there must be $j_2\geq i$ such that $E_{j_2}$ has two distinct elements of order equal to $d$. This implies that 
$$H_{E_{j_2}}(d)\leq H_{E_{j_2}}(d-1)^{\la d-1\ra}-2,$$
and so (by induction, Macaulay's theorem, Lemma~\ref{technical}, and \eqref{notwish})
$$H_E(d)\leq \sum_{j=1}^nH_{E_j}(d-1)^{\la d-1\ra}-2\leq H_{\bar\mu|_{\ell}}(d).$$
We have reached the desired contradiction.

\medskip

\noindent \underline{Case 3: $H< d\leq \ord\mu^{(n)}_{\l_\tau}$.} Assume, towards a contradiction, that $H_E(d)>H_{\bar\mu|_\ell}(d)$. We claim that then
\begin{equation}\label{first}
H_E(d-1)=H_{\bar\mu|_\ell}(d-1).
\end{equation}
and
\begin{equation}\label{find}
H_{E_j}(d)=H_{E_j}(d-1)^{\la d-1\ra}\quad \text{ for }j=1,\dots,n.
\end{equation}

We consider, as in Case 2, two cases. First assume $d\neq \ord\mu_1^{(i)}$ for any $i=1,\dots,n$. Again we will use \eqref{wish}. If \eqref{first} does not hold, induction, together with Macaulay's theorem and Lemma~\ref{technical}, yield the following contradiction
$$H_E(d)\leq \sum_{j=1}^nH_{E_j}(d-1)^{\la d-1\ra}\leq  H_{\bar\mu|_\ell}(d)< H_E(d)$$
On the other hand, if \eqref{find} does not hold then induction yields the following contradiction
$$H_E(d)\leq \sum_{j=1}^nH_{E_j}(d-1)^{\la d-1\ra}-1\leq  H_{\bar\mu_\ell}(d)< H_E(d).$$

Now assume $d=\ord\mu_1^{(i)}$ for some $i$. Assume, towards a contradiction, that \eqref{first} does not hold. Then, by induction,
$$H_E(d)\leq \sum_{j=1}^nH_{E_j}(d-1)^{\la d-1\ra}\leq H_{\bar\mu_\ell}(d)+1\leq H_E(d).$$
Thus all the above inequalities become equality, and so
$$\sum_{j=1}^nH_{E_j}(d-1)^{\la d-1\ra}+1 = H_{\bar\mu^{(n)}|_{\l_\tau}}(d-1)^{\la d-1\ra}+\sum_{j=1}^{n-1}H_{\bar\mu^{(j)}}(d-1)^{\la d-1\ra}.$$ 
Corollary \ref{techcon} now yields that $\binom{m-2+d}{d-1}=1$ but this is impossible since $m\geq 2$ and $d>1$. On the other hand, assume, towards a contradiction, that \eqref{find} does not hold. By induction,
$$H_E(d)\leq \sum_{j=1}^nH_{E_j}(d-1)^{\la d-1\ra}-1\leq  H_{\bar\mu_\ell}(d)+1\leq H_E(d).$$
Thus, the above are all equalities, and, as in Case 2, Corollary \ref{techcon} yields that (after possibly reordering the variables $x_j$'s)
$$H_{E_j}(d-1)=0\; \text{ for }j<i, \; \text{ and } \; H_{E_j}(d-1)=\binom{m-2+d}{d-1} \; \text{ for }j\geq i.$$
But now, since $d>H$ and so there are no elements in $E$ of order $d$, we get $H_{E_j}(d)=H_{E_j}(d-1)^{\la d-1\ra}$ for all $j$, which is the desired contradiction. 

We have thus shown \eqref{first} and \eqref{find}. Since $d>H$, Gotzmann's persistence theorem (see Theorem \ref{Gotzmann}) shows that for all $s\geq d$ we have 
$$H_{E_j}(s)=H_{E_j}(s-1)^{\la s-1\ra},\quad \text{ for all } j.$$
This shows that if we choose $\ell'$ to be maximal such that $d-1=\ord \mu_{\ell'}$, then $H_E(s)=H_{\bar\mu|_{\ell'}}(s)$ for all $s\geq d$, where 
$$H_{\bar\mu|_{\ell'}}(s):=H_{\bar\mu^{(1)}}(s)+\cdots+H_{\bar\mu^{(i-1)}}(s)+H_{\bar\mu^{(i)}|_{\l'}}(s)$$
and $i$ and $\l'$ are such that $\bar\mu|_{\ell'}$ is the concatenation of $\bar\mu^{(1)},\dots,\bar\mu^{(i-1)},\bar\mu^{(i)}|_{\l'}$.
This implies that $\sum_{j=1}^n\omega_{E_j}$ and $\omega_{\bar\mu^{(i)}|_{\l'}}+\omega_{\bar\mu^{(i-1)}}+\cdots+\omega_{\bar\mu^{(1)}}$ differ by a constant. Note that $\ell'<\ell$, and so, by Proposition \ref{shape}, the former numerical polynomial has degree strictly larger than $\tau$ (and consequently so does the former polynomial). This contradicts the fact that $\P$ has differential type $\tau$. This finishes the proof of Case~3.

\medskip

\noindent \underline{Case 4: $d>\ord\mu^{(n)}_{\l_\tau}$.} In this case we have \eqref{finalwish}, and so, by induction, Macaulay's theorem and Fact \ref{technical}, we get
$$H_E(d)\leq \sum_{j=1}^nH_{E_j}(d-1)^{\la d-1\ra}\leq H_{\bar\mu_\ell}(d).$$
This completes the proof.
\end{proof}

In the following corollary by the \emph{volume of $\bar\mu$} we mean 
$$\Vol\bar \mu=\Vol\bar\mu^{(1)}+\cdots+\Vol\bar\mu^{(n)}.$$
We also recall that, for $0\leq\tau<m$, we denote by $\M(r,m,n,\tau)$ the unique nonzero entry of $\mu^{(n)}_{\l_\tau}$.

\begin{corollary}\label{finally}
Let $\Sigma\subset K\{x_1,\dots,x_n\}$ be of order at most $r$ and $\P$ a prime component. If $\P$ has differential type $0<\tau<m$, then the typical differential dimension $a_\tau$ of $\P$ satisfies
$$a_\tau\leq \M(r,m,n,\tau)$$
Furthermore, if $\tau=0$, then 
$$a_0\leq \Vol (\bar \mu)$$
\end{corollary}
\begin{proof}
First suppose $0<\tau<m$. Recall that $\mu^{(n)}_{\l_\tau}$ is of the form
$$((0,\dots,0,\M(r,m,n,\tau),0,\dots,0),1)$$
where $\M(r,m,n,\tau)\neq 0$ is in the $(m-\tau)$-th position. By Proposition~\ref{shape}, we have that
$$\omega_{\bar\mu^{(n)}|_{\l_\tau}}(t)=\M(r,m,n,\tau)\binom{t+\tau}{\tau}+c$$
Since all the other polynomials $\omega_{\bar\mu^{(1)}},\dots,\omega_{\bar\mu^{(n-1)}}$ are constants (again by Proposition~\ref{shape}), Theorem~\ref{main} yields
$$a_\tau\leq \M(r,m,n,\tau).$$
Now assume $\tau=0$. In this case $\l_\tau=L_n$ (the length of $\bar\mu^{(n)}$) and so $\bar\mu^{(n)}|_{\l_\tau}=\bar\mu^{(n)}$. By Proposition \ref{shape}, $\omega_{\bar\mu^{(j)}}=\Vol\bar\mu^{(j)}$ for $j=1,\dots,n$. Thus, by Theorem \ref{main}, $a_0\leq \Vol\bar\mu$.
\end{proof}

\begin{remark}
When $n=1$ and $\tau=m-1$, we get $\mu_{\l_\tau}=(r,0,\dots,0)$, and so the corollary yields $a_{m-1}\leq r$. Thus, in this case, we recover Kolchin's result (see Fact~\ref{introfact}).
\end{remark}

\section{Some computations}\label{compu}

In this section we provide recursive algorithms and formulas that compute the value of $\M(r,m,n,\tau)$ and $\Vol\bar\mu$. By Corollary \ref{finally}, this will yield effectively computable upper bounds for the typical differential dimension of any prime component, of differential type $\tau$, of a differential system of order at most $r$ in $n$ variables and $m$ derivations.  We assume the notation and terminology used in previous sections.

Note that, by Corollary \ref{mainrough}, in the case $m=1$ we already have optimal upper bounds at our disposal. Namely, $a_1+a_0\leq nr$, where $a_1$ and $a_0$ are the standard coefficients of the Kolchin polynomial of $\P$; that is, $\omega_\P(t)=a_1(t+1)+a_0$.  Furthermore, for arbitrary $m$, Fact \ref{introfact} tells us that if $\tau=m$ then $a_\tau\leq n$, and if $\tau=m-1$ then $a_{\tau}\leq nr$. Both of these bounds are optimal. Thus, in this section, we focus in the cases when $0\leq \tau\leq m-2$ for $m\geq 2$. We also assume $r>0$ (the case $r=0$ is trivial as, under the assumption $\tau<m$, we get $\omega_\P=0$ ).

\subsection{The case $m=2$}
By the above reductions, we only need to consider the case when $\tau=0$. By Corollary~\ref{finally}, we have $a_0\leq \Vol(\bar\mu)$. To compute the latter note that $\bar\mu^{(1)}$ is given by
$$(r,0),(r-1,1),(r-2,3),\dots, (1,2r-3), (0,2r-1)$$
inside the $n$-th copy of $\N^m$; i.e., $\N^m\times\{n\}$. So $\Vol\bar\mu^{(1)}=r^2$. Similary, inside the $(n-1)$-th copy of $\N^m$, $\bar\mu^{(2)}$ is given by
$$(2r,0),(2r-1,1),(2r-2,3),\dots, (1,4r-3), (0,4r-1)$$
and so $\Vol\bar\mu^{(2)}=4r^2$. Continuing in this fashion, we find that $\Vol\bar\mu^{(i)}=4^{i-1}r^2$ for $i=1,\dots,n$. Therefore
$$\Vol\bar\mu=\sum_{i=0}^{n-1}4^ir^2=\left(\frac{4^n-1}{3}\right)r^2.$$
We can conclude
\begin{corollary}\label{final2}
Suppose $m=2$, that is $\D=\{\d_1,\d_2\}$. Let $\Sigma\subset K\{x_1,\dots,x_n\}$ be of order at most $r$ and $\P$ a prime component. If $\P$ has differential type $\tau=0$, then its typical differential dimension satisfies
$$a_0\leq \left(\frac{4^n-1}{3}\right)r^2.$$
\end{corollary}

\medskip

\begin{remark}
In the case when $n=1$ the above bound reduces to $a_0\leq r^2$. By Example \ref{introex}, this is optimal. So now we ask the question: is the bound of Corollary~\ref{final2} optimal for $n\geq 2$? This question remains open and it is the subject of an ongoing research project.
\end{remark}

\subsection{The case $m\geq 3$}

We fix the differential type $0\leq \tau\leq m-2$. Recall that the first element of the sequence $\bar\mu^{(n)}$ inside of $\N^m\times\{1\}$ has the form $(C_{r,m}^{n-1},0,\dots,0)$, where $C_{r,m}^0:=r$ and for $n>0$, as in Section \ref{rough}, 
$$C_{0,m}^1=0, \quad\; C_{r,m}^1=A(m-1,C_{r-1,m}^1), \quad \text{ and } \quad C_{r,m}^n=C_{C_{r,m}^{n-1},m}^1,$$
where $A$ is the Ackermann function.

Due to the recursive nature of the construction of the sequence $\bar\mu^{(n)}$, the following algorithm yields the \emph{unique nonzero entry} of $\bar\mu^{(n)}_{\l_\tau}$; in other words, $\M(r,m,n,\tau)$. 

Let $\Omega_{r,m}^\tau:\N^{m}\to \N$ be given by
$$
\Omega_{r,m}^\tau(0,\dots,0,u_{m-\tau},u_{m-\tau-1},\dots,u_m)=u_{m-\tau}
$$
with
\begin{align*}
\Omega_{r,m}^\tau&(u_1,\dots,u_s,0,\dots,0,u_m)\\&=\Omega_{r,m}^\tau(u_1,\dots,u_s-1,u_m+2,0,\dots,0),\quad s<m-1, \ u_s>0
\end{align*}
 and 
\begin{align*}
\Omega_{r,m}^\tau&(u_1,\dots,u_{m-1},u_m)\\&=\Omega_{r,m}^\tau(u_1,\dots,u_{m-1}-1,u_m+2),\quad u_{m-1}>0. 
\end{align*}
Then, $\M(r,m,n,\tau)=\Omega_{r,m}^\tau(C_{r,m}^{n-1}-1,1,0,\dots,0)$.

By Remark \ref{vol}, the following algorithm yields the volume of $\bar\mu$.

Let $\Upsilon_{r,m}:\N\times\N^{m}\to \N$ be given by
$$
\Upsilon_{r,m}(j,(0,\dots,0,u_m))=j
$$
with
\begin{align*}
\Upsilon_{r,m}&(j,(u_1,\dots,u_s,0,\dots,0,u_m))\\&=\Upsilon_{r,m}(j,(u_1,\dots,u_s-1,u_m+2,0,\dots,0)),\quad s<m-1, \ u_s>0
\end{align*}
 and 
\begin{align*}
\Upsilon_{r,m}&(j,(u_1,\dots,u_{m-1},u_m))\\&=\Upsilon_{r,m}(j+u_m+2,(u_1,\dots,u_{m-1}-1,u_m+2)),\quad u_{m-1}>0. 
\end{align*}

Then, for each $i=1,\dots,n$, 
$$\Vol\bar\mu^{(i)}=\Upsilon_{r,m}(0,(C_{r,m}^{i-1}-1,1,0,\dots,0)),$$
and so
$$\Vol\bar\mu=\sum_{i=1}^n \Upsilon_{r,m}(0,(C_{r,m}^{i-1}-1,1,0,\dots,0)).$$

We will now give more explicit formulas using the Ackermann function and $C_{r,m}^n$. Let $B_{r,m,n}^1=C_{r,m}^{n-1}$,
$$B_{r,m,n}^2=C_{B_{r,m,n}^1-1,m}^1+1,$$ 
and for $3\leq i \leq m$
\begin{equation}\label{compbound}
B_{r,m,n}^i=A(m-i+2,B_{r,m,n}^{i-1}-1)+1
\end{equation}
By \cite[Corollary 7.4]{MS}, we have
$$\M(r,m,n,\tau)=B_{r,m,n}^{m-\tau}.$$

In order to give a formula for the volume of $\bar\mu$, we extend the domain of the Ackermann function to include points of  the form $(x,-1)$ for nonnegative integers $x$. Following \cite[\S2]{MS}, we set $A(0,1)=0$, and $A(x,-1)=1$ for $x>0$. Now let $F(0,y)=1$ and 
\begin{equation}\label{comp1}
F(x,y)=\sum_{i=-1}^{y-1}F(x-1,A(x,i))
\end{equation}
Also, define $\nu(x,y)$ by $\nu(x,0)=0$ and, for $x,y>0$,
\begin{equation}\label{comp2}
\nu(x,y)=F(x-1,C_{y-1,x}^1+1)+\nu(x,y-1).
\end{equation}
Then, by \cite[Corollary 7.7]{MS}, we have that for each $i=1,\dots,n$
$$\Vol\bar\mu^{(i)} =\nu(m,C_{r,m}^{i-1}),$$
 and so
 $$\Vol\bar\mu=\sum_{i=1}^n \nu(m,C_{r,m}^{i-1}).$$

Let us state the above results as
\begin{corollary}\label{question}
Suppose $m\geq 3$. Let $\Sigma\subset K\{x_1,\dots,x_n\}$ be of order at most $r$ and $\P$ a prime component. If $\P$ has differential type $1\leq\tau\leq m-2$, then its typical differential dimension satisfies
$$a_\tau\leq B_{r,m,n}^{m-\tau}$$
with $B_{r,m,n}^{i}$ given in \eqref{compbound}. Furthermore, if $\tau=0$ then
$$a_0\leq \sum_{i=1}^n \nu(m,C_{r,m}^{i-1})$$
with $\nu$ given in \eqref{comp1} and \eqref{comp2}.
\end{corollary}

\begin{example}
Using the above formulas, we do the computations in the case when $m=3$ and $n=1$. By the above reductions, in this case we only have to consider the cases when $\tau=0$ or $1$. When $\tau=1$, from the formula of $B_{r,m,n}^{m-\tau}$, we get
$$a_1\leq C_{r-1,3}^1+1=3\cdot 2^{r-1}-2$$
When $\tau=0$, from the formula of $\nu(m,-)$, we get
$$a_0\leq \sum_{i=0}^{r-1}(C_{i,3}^1+1)(C_{i,3}^1+2)=\sum_{i=0}^{r-1}9(4^{i}-2^{i+1}+2^i)+2.$$
We note that these bounds are indeed a significant improvement from the bounds found in Corollary \ref{mainrough}. There we found that 
$$a_1\leq \frac{1}{4}(27(2^r-1)^3+27(2^r-1)^2+6(2^r-1))^2$$ 
and 
$$a_0\leq \frac{1}{8}(27(2^r-1)^3+27(2^r-1)^2+6(2^r-1))^3.$$
\end{example}

\medskip

We finish with the following important question.

\begin{question}
Are the upper bounds found in Corollary \ref{question} optimal? 
\end{question}

Even answering this question in the special case $n=1$ seems to be a difficult but rather interesting task.

\bibliographystyle{plain}

\end{document}